\documentclass[12pt]{amsart}
\usepackage{amssymb,latexsym}
\input epsf
\usepackage{graphicx}
\newtheorem{theorem}{Theorem}[section]
\newtheorem{proposition}[theorem]{Proposition}
\newtheorem{corollary}[theorem]{Corollary}
\newtheorem{definition}[theorem]{Definition}
\newtheorem{remark}[theorem]{Remark}
\newtheorem{lemma}[theorem]{Lemma}
\newtheorem{example}[theorem]{Example}

\newcommand{\des}{{\rm des}}
\newcommand{\ex}{{\rm ex}}
\newcommand{\link}{{\rm link}}
\newcommand{\nc}{{\rm NC}}

\newcommand{\sd}{{\rm sd}}

\newcommand{\dD}{{\mathcal D}}
\newcommand{\eE}{{\mathcal E}}

\newcommand{\jJ}{{\mathcal J}}
\newcommand{\sS}{{\mathcal S}}
\newcommand{\xX}{{\mathcal X}}

\renewcommand{\to}{\rightarrow}

\newcommand{\sm}{{\smallsetminus}}
\begin{document}
\title[The local $h$-vector of the cluster subdivision]
{The local $h$-vector of the cluster subdivision of a simplex}

\author{Christos~A.~Athanasiadis}
\author{Christina~Savvidou}

\address{Department of Mathematics
(Division of Algebra-Geometry)\\
University of Athens\\
Panepistimioupolis\\
15784 Athens, Greece}
\email{caath@math.uoa.gr, savvtina@math.uoa.gr}

\date{}
%
\begin{abstract}
The cluster complex $\Delta (\Phi)$ is an abstract simplicial complex,
introduced by Fomin and Zelevinsky for a finite root system $\Phi$.
The positive part of $\Delta (\Phi)$ naturally defines a simplicial
subdivision of the simplex on the vertex set of simple roots of $\Phi$.
The local $h$-vector of this subdivision, in the sense of Stanley, is
computed and the corresponding $\gamma$-vector is shown to be
nonnegative. Combinatorial interpretations to the entries of the local
$h$-vector and the corresponding $\gamma$-vector are provided for the
classical root systems, in terms of noncrossing partitions of types $A$
and $B$. An analogous result is given for the barycentric subdivision
of a simplex.

\bigskip
\noindent
\textit{Key words and phrases}. Local $h$-vector, barycentric
subdivision, cluster complex, cluster subdivision, $\gamma$-vector,
noncrossing partition.
\end{abstract}

\maketitle

\section{Introduction and results}
\label{intro}

Local $h$-vectors were introduced by Stanley \cite{Sta92} as a
fundamental tool in his theory of face enumeration for
subdivisions of simplicial complexes. Given a (finite,
topological) simplicial subdivision $\Gamma$ of the abstract
simplex $2^V$ on an $n$-element vertex set $V$, the local
$h$-polynomial $\ell_V (\Gamma, x)$ is defined as an alternating
sum of the $h$-polynomials of the restrictions of $\Gamma$ to the
faces of $2^V$ (see Section~\ref{sec:back} for all relevant
definitions). The local $h$-vector of $\Gamma$ is the sequence of
coefficients $\ell_V (\Gamma) = (\ell_0, \ell_1,\dots,\ell_n)$,
where $\ell_V (\Gamma, x) = \ell_0 + \ell_1 x + \cdots + \ell_n
x^n$.

The importance of local $h$-vectors stems from their appearance in the
locality formula \cite[Theorem~3.2]{Sta92}, which expresses the $h$-polynomial
of a simplicial subdivision of a pure simplicial complex $\Delta$ as a sum
of local contributions, one for each face of $\Delta$. Several fundamental
properties of local $h$-vectors, including symmetry for all topological
subdivisions, nonnegativity for quasi-geometric subdivisions and
unimodality for regular (geometric) subdivisions, were proven in
\cite{Sta92}.

The local $h$-vector of the barycentric subdivision of a simplex affords
an elegant interpretation \cite[Proposition~2.4]{Sta92} in terms of the
combinatorics of permutations. The focus of this paper is on another
example of subdivision of the simplex with remarkable combinatorial
properties, termed as the \textit{cluster subdivision}. This is the
simplicial subdivision of the simplex on the vertex set of simple roots
of a finite root system
$\Phi$ which is naturally defined by the positive part of the cluster
complex $\Delta(\Phi)$ \cite{FZ03} (see the discussion below). Our main
results compute the local $h$-vector of the cluster subdivision, providing
combinatorial interpretations for the classical root systems in terms of
the combinatorics of noncrossing partitions (for the deep connections
between cluster combinatorics and noncrossing partitions see, for
instance, \cite{ABMW06, Ch04, Re07}).

Before proceeding further, we recall the following notation and
terminology from \cite{Ath12}. Let $\Gamma$ be a simplicial
subdivision of an $(n-1)$-dimensional simplex $2^V$. Since $\ell_V
(\Gamma, x)$ has symmetric coefficients, there exists
\cite[Proposition~2.1.1]{Ga05} a unique polynomial $\xi_V (\Gamma,
x) = \xi_0 + \xi_1 x + \cdots + \xi_{\lfloor n/2 \rfloor}
x^{\lfloor n/2 \rfloor}$ such that
  $$ \ell_V (\Gamma, x) \, = \, \sum_{i=0}^{\lfloor n/2 \rfloor} \, \xi_i
     x^i (1+x)^{n-2i}. $$
Following \cite[Section~5]{Ath12}, we will refer to $\xi_V (\Gamma, x)$
as the \emph{local $\gamma$-polynomial} of $\Gamma$ (with respect to $V$)
and to the sequence $\xi_V (\Gamma) = (\xi_0, \xi_1,\dots,\xi_{\lfloor
n/2 \rfloor})$ as the \emph{local $\gamma$-vector} of $\Gamma$ (with
respect to $V$). As explained in \cite{Ath12}, these concepts play a
role in the theory of face enumeration for flag homology spheres and their
flag simplicial subdivisions.

We will show that $\xi_V (\Gamma, x)$ has nonnegative coefficients for all
cluster subdivisions by providing combinatorial interpretations in terms of 
noncrossing partitions, or by explicit computation. Since cluster 
subdivisions are geometric and flag, this result provides evidence for a 
conjecture by the first author \cite[Conjecture~5.4]{Ath12}, stating that 
$\xi_V (\Gamma, x)$ has nonnegative coefficients for a family of simplicial 
subdivisions of the simplex which includes all flag geometric subdivisions. 
We will also provide combinatorial interpretations to the coefficients of 
$\xi_V (\Gamma, x)$ for the barycentric subdivision of the simplex.

The remainder of this section states the main results of this paper
in more precise form. Their proofs are given in Sections~\ref{sec:cproof}
and \ref{sec:bproof}, after some of the relevant background on simplicial
complexes, cluster complexes, simplicial subdivisions and noncrossing
partitions is recalled in Section~\ref{sec:back}. Remarks and related
open problems are included in Section~\ref{sec:rem}.

\subsection{Cluster subdivisions} \label{subsec:cs}
Let $\Phi$ be a finite root system of rank $n$, equipped with a
positive system $\Phi^+$ and corresponding simple system $\Pi = \{
\alpha_i: i \in I\}$, where $I$ is an $n$-element index set. The
cluster complex $\Delta (\Phi)$ was introduced by Fomin and
Zelevinsky in the context of algebraic $Y$-systems \cite{FZ03}. It
is an abstract simplicial complex on the vertex set $\Phi^+ \cup
(-\Pi)$, consisting of the positive roots and the negative simple
roots, which is homeomorphic to the $(n-1)$-dimensional sphere. When
$\Phi$ is crystallographic, the combinatorics of $\Delta (\Phi)$
encodes the exchange of clusters in the corresponding cluster
algebra of finite type \cite{FZ02}. An overview of cluster complexes
and their connection to cluster algebras can be found in
\cite{FR07}. The restriction $\Delta_+ (\Phi)$ of $\Delta (\Phi)$ on
the vertex set $\Phi^+$, known as the positive part of $\Delta
(\Phi)$, is homeomorphic to the $(n-1)$-dimensional ball.

The complex $\Delta_+ (\Phi)$ has the structure of a (geometric)
simplicial subdivision of the simplex $2^\Pi$ on the vertex set
$\Pi$ (see Section~\ref{subsec:cluster}). The restriction of this
subdivision to the face $\{\alpha_i: i \in J\}$ of $2^\Pi$ indexed
by $J \subseteq I$ is the complex $\Delta_+ (\Phi_J)$, where
$\Phi_J$ is the standard parabolic root subsystem of $\Phi$
corresponding to $J$ (so that $\Phi_I = \Phi$). We will refer to
this subdivision as the cluster subdivision associated to $\Phi$
and will denote it by $\Gamma(\Phi)$. We will write
  \begin{equation} \label{eq:defclocalh}
    \ell_I (\Gamma (\Phi), x) \ = \ \sum_{i=0}^n \, \ell_i (\Phi) x^i
  \end{equation}
for the local $h$-polynomial of $\Gamma(\Phi)$ and $\ell_I (\Phi)
= (\ell_0 (\Phi), \ell_1 (\Phi),\dots,\ell_n (\Phi))$ for the
corresponding local $h$-vector. The relevant definitions lead (see
Section~\ref{subsec:cluster}) to the formula
  \begin{equation} \label{eq:clocalh}
    \ell_I (\Gamma (\Phi), x) \ = \ \sum_{J \subseteq I} \ (-1)^{|I \sm J|} \
    h (\Delta_+ (\Phi_J), x),
  \end{equation}
where $h (\Delta_+ (\Phi_J), x)$ is the $h$-polynomial of $\Delta_+ (\Phi_J)$.
The results of \cite{Sta92}, mentioned earlier, imply that $\ell_I (\Gamma
(\Phi), x)$ has nonnegative and symmetric coefficients for every root system
$\Phi$.

The $h$-polynomial of $\Delta_+ (\Phi)$ admits several combinatorial
interpretations \cite[Corollary~1.4 and Theorem~1.5]{AT06} \cite[Corollaries~7.4
and 7.5]{ABMW06} in terms of order ideals of roots, hyperplane regions,
Weyl group orbits on a finite torus, lattice points and noncrossing partitions.
It was computed explicitly for all irreducible (crystallographic) root systems
in \cite[Section~6]{AT06}. We denote by $\nc^A (n)$ and $\nc^B (n)$ the set of
noncrossing partitions of the set $\{1, 2,\dots,n\}$ and that of
$B_n$-noncrossing partitions, respectively, and refer to Section~\ref{subsec:nc}
for the relevant background and any undefined terminology. Our first result
determines the local $h$-polynomial of $\Gamma (\Phi)$ as follows.

\begin{theorem} \label{thm:clocalh}
Let $\ell_I (\Gamma (\Phi), x) = \sum_{i=0}^n \ell_i (\Phi) x^i$ be the local
$h$-polynomial of the cluster subdivision $\Gamma (\Phi)$, associated to an
irreducible root system $\Phi$ of rank $n$ and Cartan-Killing type $\xX$. Then
$\ell_i (\Phi)$ is equal to:
 \begin{itemize}
   \item[$\bullet$] the number of partitions $\pi \in \nc^A (n)$ with $i$ blocks,
                    such that every singleton block of $\pi$ is nested, if $\xX =
                    A_n$,
   \item[$\bullet$] the number of partitions $\pi \in \nc^B (n)$ with no zero
                    block and $i$ pairs $\{B, -B\}$ of nonzero blocks, such that
                    every positive singleton block of $\pi$ is nested, if
                    $\xX = B_n$,
   \item[$\bullet$] $n-2$ times the number of partitions $\pi \in \nc^A (n-1)$
                    with $i$ blocks, if $\xX = D_n$.
  \end{itemize}
Moreover, $\ell_I (\Gamma (\Phi), x)$ is equal to
 $$ \begin{cases}
    (m-2)x, & \text{if \ $\xX = I_2(m)$} \\
    8x + 8x^2, & \text{if \ $\xX = H_3$} \\
    42x + 124x^2 + 42x^3, & \text{if \ $\xX = H_4$} \\
    10x + 29x^2 + 10x^3, & \text{if \ $\xX = F_4$} \\
    7x + 63x^2 + 125x^3 + 63x^4 + 7x^5, & \text{if \ $\xX = E_6$} \\
    16x + 204x^2 + 644x^3 + 644x^4 + 204x^5 + 16x^6, & \text{if \ $\xX = E_7$} \\
    44x + 748x^2 + 3380x^3 + 5472x^4 + 3380x^5 + 748x^6 + 44x^7, & \text{if \
    $\xX = E_8$}.  \end{cases} $$
\end{theorem}

\bigskip
We will write $\xi_I (\Phi) = (\xi_0 (\Phi), \xi_1 (\Phi),\dots,\xi_{\lfloor n/2
\rfloor} (\Phi))$ for the local $\gamma$-vector of $\Gamma(\Phi)$, so that
  \begin{equation} \label{eq:defclocalg}
    \ell_I (\Gamma (\Phi), x) \ = \ \sum_{i=0}^{\lfloor n/2 \rfloor} \, \xi_i
    (\Phi) \, x^i (1+x)^{n-2i}.
  \end{equation}

Our second result computes the numbers $\xi_i (\Phi)$ (hence, via
equation~(\ref{eq:defclocalg}), the numbers $\ell_i (\Phi)$ as
well) explicitly.

\begin{theorem} \label{thm:clocalg}
Let $\Phi$ be an irreducible root system of rank $n$ and Cartan-Killing type $\xX$
and let $\xi_i (\Phi)$ be the integers uniquely defined by (\ref{eq:defclocalg}).
Then $\xi_0 (\Phi) = 0$ and
  $$ \xi_i (\Phi) \ = \ \begin{cases}
    \displaystyle \frac{1}{n-i+1} \binom{n}{i} \binom{n-i-1}{i-1},
    & \text{if \ $\xX = A_n$} \\ & \\
    \displaystyle \binom{n}{i} \binom{n-i-1}{i-1}, & \text{if \ $\xX = B_n$} \\
    & \\ \displaystyle \frac{n-2}{i} \binom{2i-2}{i-1} \binom{n-2}{2i-2},
    & \text{if \ $\xX = D_n$}  \end{cases} $$

\bigskip
\noindent
for $1 \le i \le \lfloor n/2 \rfloor$. Moreover,

 $$ \sum_{i=0}^{\lfloor n/2 \rfloor} \, \xi_i (\Phi) x^i \ = \ \begin{cases}
    (m-2)x, & \text{if \ $\xX = I_2(m)$} \\
    8x, & \text{if \ $\xX = H_3$} \\
    42x + 40x^2, & \text{if \ $\xX = H_4$} \\
    10x + 9x^2, & \text{if \ $\xX = F_4$} \\
    7x + 35x^2+ 13x^3, & \text{if \ $\xX = E_6$} \\
    16x + 124x^2 + 112x^3, & \text{if \ $\xX = E_7$} \\
    44x + 484x^2 + 784x^3 + 120x^4, & \text{if \ $\xX = E_8$}.  \end{cases} $$
\end{theorem}

\medskip
The proof of Theorem~\ref{thm:clocalg}, given in Section~\ref{sec:cproof},
shows that when $\xX = A_n$ (respectively, $\xX = B_n$), the numbers $\xi_i 
(\Phi)$ enumerate partitions $\pi \in \nc^A (n)$ (respectively, partitions $\pi 
\in \nc^B (n)$ with no zero block) which have no singleton block, by the number 
of blocks; see Propositions~\ref{prop:An} and \ref{prop:Bn}.

When $\Phi$ is crystallographic, the cluster complex $\Delta(\Phi)$ can be
realized as the boundary complex of a simplicial convex polytope \cite{CFZ02}.
One may deduce from this statement that $\Gamma(\Phi)$ is a regular (geometric)
subdivision of the simplex $2^\Pi$. Thus \cite[Theorem~5.2]{Sta92} implies that
the local $h$-vector of $\Gamma(\Phi)$ is unimodal, i.e., $\ell_0 (\Phi) \le
\ell_1 (\Phi) \le \cdots \le \ell_{\lfloor n/2 \rfloor} (\Phi)$. The following
corollary of Theorem~\ref{thm:clocalg} provides a stronger statement.

\begin{corollary} \label{cor:clocalg}
For every root system $\Phi$ the local $\gamma$-vector of $\Gamma(\Phi)$ is
nonnegative, i.e., we have $\xi_i (\Phi) \ge 0$ for every index $i$.
\end{corollary}

\subsection{Barycentric subdivisions} \label{subsec:bs}
Let $V$ be an $n$-element set. We denote by $\sd(2^V)$ the (first) barycentric
subdivision of the simplex $2^V$ and by $\sS_n$ the set of permutations of $\{1,
2,\dots,n\}$. We recall that for $w \in \sS_n$, a descent of $w$ is an index $1
\le i \le n-1$ such that $w(i) > w(i+1)$; an excedance of $w$ is an index $1 \le
i \le n$ such that $w(i) > i$. The local $h$-polynomial of $\sd(2^V)$ was
computed in \cite[Proposition~2.4]{Sta92} as
  \begin{equation} \label{eq:localbary}
    \ell_V (\sd(2^V), x) \ = \ \sum_{w \in \dD_n} \, x^{\ex(w)},
  \end{equation}
where $\dD_n$ is the set of derangements (permutations with no fixed points) in
$\sS_n$ and $\ex(w)$ is the number of excedances of $w \in \sS_n$. We will
provide similar combinatorial interpretations to the local $\gamma$-polynomial
of $\sd(2^V)$ after we introduce some more terminology.

For $w \in \sS_n$, an \emph{ascending run} (or simply, a run) of $w$ is a maximal
string $\{i, i+1,\dots,j\}$ of integers, such that $w(i) < w(i+1) < \cdots < w(j)$.
A \emph{double descent} of $w$ is an index $2 \le i \le n-1$ such that $w(i-1) >
w(i) > w(i+1)$; a \emph{double excedance} of $w$ is an index $1 \le i \le n$ such
that $w(i) > i > w^{-1}(i)$. A \emph{left to right maximum} of $w$ is an index $1
\le j \le n$ such that $w(i) < w(j)$ for all $1 \le i < j$.
\begin{theorem} \label{thm:bary}
  Let $(\xi_0, \xi_1,\dots,\xi_{\lfloor n/2 \rfloor})$ be the local $\gamma$-vector
  of the barycentric subdivision $\sd(2^V)$ of the $(n-1)$-dimensional simplex $2^V$.
  Then $\xi_i$ is equal to each of the following:
  \medskip
    \begin{itemize}
      \item[(i)] the number of permutations $w \in \sS_n$ with $i$ runs and no run
      of length one,
      \item[(ii)] the number of derangements $w \in \dD_n$ with $i$ excedances and
      no double excedance,
      \item[(iii)] the number of permutations $w \in \sS_n$ with $i$ descents and
      no double descent, such that every left to right maximum of $w$ is a descent.
  \end{itemize}
In particular, we have $\xi_i \ge 0$ for all $0 \le i \le \lfloor n/2 \rfloor$.
\end{theorem}

\medskip
For the first few values of $n$ we have:

  $$ \sum_{i=0}^{\lfloor n/2 \rfloor} \, \xi_i x^i \ = \ \begin{cases}
    x, & \text{if \ $n=2, 3$} \\
    x + 5x^2, & \text{if \ $n=4$} \\
    x + 18x^2, & \text{if \ $n=5$} \\
    x + 47x^2+ 61x^3, & \text{if \ $n=6$} \\
    x + 108x^2 + 479x^3, & \text{if \ $n=7$} \\
    x + 233x^2 + 2414x^3 + 1385x^4, & \text{if \ $n=8$} \\
    x + 486x^2 + 9970x^3 + 19028x^4, & \text{if \ $n=9$}.  \end{cases} $$

\medskip
The right-hand side of (\ref{eq:localbary}) is known as the \emph{derangement
polynomial} of order $n$; see, for instance, \cite[Section~1]{CTZ09}, where some 
of its basic properties are summarized. Theorem~\ref{thm:bary} gives a 
combinatorial proof of the unimodality of this polynomial, thus answering a 
question of Brenti \cite{Bre90}; see the third comment in Section~\ref{sec:rem}. 
Moreover, it implies that for given $n$, the sum of the coefficients $\xi_i$ is 
equal to the total number of permutations in $\sS_n$ with no ascending run of 
length one. Such permutations have been considered (in a more general context) 
and enumerated by Gessel \cite[Chapter~5]{Ge77}.

We should point out that the nonnegativity of the numbers $\xi_i$ follows from 
the fact that the derangement polynomials are (symmetric and) real-rooted 
\cite{Zh95}. Alternatively, this can be deduced from \cite[Proposition 6.1]{Ath12}, 
which proves the nonnegativity of the local $\gamma$-vector for a family of flag
simplicial subdivisions which can be obtained from the trivial subdivision of 
a simplex by successive stellar subdivisions.  

\section{Subdivisions, clusters and noncrossing partitions}
\label{sec:back}

This section begins by recalling basic definitions on simplicial complexes,
simplicial subdivisions and their enumerative invariants. Cluster complexes are
then reviewed and cluster subdivisions are formally defined. The section ends
with a brief discussion of noncrossing partitions of types $A$ and $B$. More
information on these topics can be found in \cite{Bj95, FR07, Re97, StaCCA} and
references therein. Throughout this paper, $|S|$ denotes the cardinality, and
$2^S$ the set of all subsets, of a finite set $S$.

\subsection{Simplicial complexes} \label{subsec:complexes}
Given a finite set $\Omega$, an (abstract) simplicial complex on the ground
set $\Omega$ is a collection $\Delta$ of subsets of $\Omega$ such that $F
\subseteq G \in \Delta$ implies $F \in \Delta$. The elements of $\Delta$ are
called \emph{faces}. The dimension of a face $F$ is defined as one less than
the cardinality of $F$. The dimension of $\Delta$ is the maximum dimension of
a face and is denoted by $\dim (\Delta)$. Faces of $\Delta$ of dimension zero
are called \emph{vertices}. A face of $\Delta$ which is maximal with respect
to inclusion is called a \emph{facet}. The simplicial complex $\Delta$ is
said to be \emph{pure} if all its facets have the same dimension.
The \emph{link} of the face $F \in \Delta$ is the subcomplex of $\Delta$ defined
as $\link_\Delta (F) = \{ G \sm F: G \in \Delta, \, F \subseteq G\}$. The
\emph{restriction} of $\Delta$ on the ground set $\Omega_0 \subseteq \Omega$ is
the subcomplex of $\Delta$ consisting of those faces which are contained in
$\Omega_0$.

Suppose that $\Omega_1$ and $\Omega_2$ are two disjoint finite sets. The (simplicial)
\emph{join} $\Delta_1 \ast \Delta_2$ of two collections $\Delta_1$ and $\Delta_2$
of subsets of $\Omega_1$ and $\Omega_2$, respectively, is the collection
whose elements are the sets of the form $F_1 \cup F_2$, where $F_1 \in \Delta_1$
and $F_2 \in \Delta_2$. The join of two (or more) simplicial complexes is again
a simplicial complex.

Every simplicial complex $\Delta$ has a geometric realization $\|\Delta\|$
\cite[Section~9]{Bj95}, uniquely defined up to homeomorphism. All
topological properties of $\Delta$ we mention in the sequel will refer to
those of $\|\Delta\|$. In particular, we say that $\Delta$ is a simplicial
(topological) ball if $\|\Delta\|$ is homeomorphic to a ball. The
\emph{boundary} of a simplicial $d$-dimensional ball $\Delta$ is the
subcomplex $\partial \Delta$, consisting of all subsets of those
$(d-1)$-dimensional faces of $\Delta$ which are contained in a unique facet
of $\Delta$. The \emph{interior} of this ball is the set $\Delta \sm \partial
\Delta$; the \emph{interior faces} are the elements of $\Delta \sm \partial
\Delta$. For example, the (abstract) simplex $2^V$, consisting of all subsets
of an $n$-element set $V$, is a simplicial $(n-1)$-dimensional ball whose only
interior face is $V$. The join of two (or more) simplicial balls is a simplicial
ball whose interior is equal to the join of the interiors of these balls.

\subsection{Simplicial subdivisions} \label{subsec:sub}
Given a finite set $V$, a (finite, topological) \textit{simplicial
subdivision} \cite[Section~2]{Sta92} of the abstract simplex $2^V$ is a
simplicial complex $\Gamma$ together with a map $\sigma: \Gamma \to 2^V$,
such that the following hold for every $F \subseteq V$: (a) the set $\Gamma_F
:= \sigma^{-1} (2^F)$ is a subcomplex of $\Gamma$ which is a simplicial ball
of dimension $\dim(F)$; and (b) the interior of $\Gamma_F$ is equal to
$\sigma^{-1} (F)$. The set $\sigma(E)$ is called the \emph{carrier} of the
face $E \in \Gamma$. The complex $\Gamma_F$ is called the \emph{restriction}
of $\Gamma$ to $F \subseteq V$.
The subdivision $\Gamma$ is called \emph{quasi-geometric} \cite[Definition~4.1
(a)]{Sta92} if there do not exist $E \in \Gamma$ and face $F \in 2^V$ of
dimension smaller than $\dim(E)$, such that the carrier of every vertex of
$E$ is contained in $F$. Moreover, $\Gamma$ is called \emph{geometric}
\cite[Definition~4.1 (b)]{Sta92} if there exists a geometric realization of
$\Gamma$ which geometrically subdivides a geometric realization of $2^V$.

Suppose that $\Gamma$ is a simplicial subdivision of the simplex $2^V$
and $\Gamma'$ is a simplicial subdivision of the simplex $2^{V'}$, where $V$
and $V'$ are disjoint sets. The join $\Gamma \ast \Gamma'$ naturally becomes
a simplicial subdivision of the simplex $2^V \ast 2^{V'} = 2^{V \cup
V'}$ if one defines the carrier of a face $E \cup E' \in \Gamma \ast \Gamma'$
as the union of the carriers of $E \in \Gamma$ and $E' \in \Gamma'$. Given
faces $F \subseteq V$ and $F' \subseteq V'$, the restriction of $\Gamma \ast
\Gamma'$ to the face $F \cup F'$ of this simplex is then equal to $\Gamma_F
\ast \Gamma'_{F'}$.

\subsection{Face enumeration}
\label{subsec:enu}

A fundamental enumerative invariant of a simplicial complex $\Delta$ is the
$h$-polynomial, defined by
  $$ h(\Delta, x) \ = \ \sum_{F \in \Delta} \ x^{|F|} (1-x)^{d-|F|}, $$
where $\dim (\Delta) = d-1$. For the join of two simplicial complexes
$\Delta_1$ and $\Delta_2$ we have $h(\Delta_1 \ast \Delta_2, x) = h(\Delta_1,
x) h(\Delta_2, x)$.

The local $h$-vector of a simplicial subdivision of a simplex was defined in
\cite[Definition~2.1]{Sta92} as follows.

\begin{definition} \label{def:localh}
{\rm Let $V$ be an $n$-element set and $\Gamma$ be a simplicial subdivision
of the simplex $2^V$. The polynomial $\ell_V (\Gamma, x) = \ell_0 + \ell_1 x +
\cdots + \ell_n x^n$ defined by
  \begin{equation} \label{eq:deflocalh}
    \ell_V (\Gamma, x) \ = \sum_{F \subseteq V} \ (-1)^{n - |F|} \
    h (\Gamma_F, x)
  \end{equation}
is the \emph{local $h$-polynomial} of $\Gamma$ (with respect to $V$). The
sequence $\ell_V (\Gamma) = (\ell_0, \ell_1,\dots,\ell_n)$ is the \emph{local
$h$-vector} of $\Gamma$ (with respect to $V$). }
\end{definition}

The local $h$-vector $\ell_V (\Gamma) = (\ell_0, \ell_1,\dots,\ell_n)$ was shown
to be symmetric (meaning that $\ell_i = \ell_{n-i}$ holds for $0 \le i \le n$)
for every simplicial subdivision $\Gamma$ of $2^V$ \cite[Theorem~3.3]{Sta92} and
to have nonnegative entries for every quasi-geometric simplicial subdivision
$\Gamma$ of $2^V$ \cite[Corollary~4.7]{Sta92}. Moreover (see
\cite[Example~2.3]{Sta92}), $\ell_0 = 0$ and $\ell_1$ is equal to the number of
interior vertices of $\Gamma$, for $n \ge 1$.

We recall from the introduction that, given a simplicial subdivision $\Gamma$ of
an $(n-1)$-dimensional simplex $2^V$, the local $\gamma$-polynomial $\xi_V (\Gamma,
x) = \xi_0 + \xi_1 x + \cdots + \xi_{\lfloor n/2 \rfloor} x^{\lfloor n/2 \rfloor}$
of $\Gamma$ (with respect to $V$) is uniquely defined by
  \begin{equation} \label{eq:defxi}
    \ell_V (\Gamma, x) \ = \ (1+x)^n \ \xi_V \left( \Gamma, \frac{x}{(1+x)^2}
    \right) \, = \, \sum_{i=0}^{\lfloor n/2 \rfloor} \, \xi_i x^i (1+x)^{n-2i}.
  \end{equation}

The following lemma will be used in the proof of Corollary~\ref{cor:clocalg}.

\begin{lemma} \label{lem:localjoin}
Let $V$ and $V'$ be disjoint finite sets. For all simplicial subdivisions $\Gamma$
of $2^V$ and $\Gamma'$ of $2^{V'}$ we have $\ell_{V \cup V'} \, (\Gamma \ast \Gamma',
x) = \ell_V (\Gamma, x) \, \ell_{V'} (\Gamma', x)$ and $\xi_{V \cup V'} \, (\Gamma
\ast \Gamma', x) = \xi_V (\Gamma, x) \, \xi_{V'} (\Gamma', x)$.
\end{lemma}
\begin{proof}
Let $n = |V|$ and $n' = |V'|$. Using the defining
equation~(\ref{eq:deflocalh}), we find that
  \begin{align*}
    \ell_{V \cup V'} \, (\Gamma \ast \Gamma', x)  & \ = \ \sum_{F \subseteq V}
    \sum_{F' \subseteq V'} \ (-1)^{|V \cup V'| - |F \cup F'|} \,
      h ( (\Gamma \ast \Gamma')_{F \cup F'}, x)  \\
    &  \\
    & \ = \ \sum_{F \subseteq V} \sum_{F' \subseteq V'} \ (-1)^{n+n' - |F| - |F'|} \,
      h (\Gamma_F \ast \Gamma'_{F'}, x)  \\
    &  \\
    & \ = \ \sum_{F \subseteq V} \sum_{F' \subseteq V'} \ (-1)^{n - |F|} \,
      h (\Gamma_F, x) \ (-1)^{n' - |F'|} \, h (\Gamma'_{F'}, x) \\
    &  \\
    & \ = \ \ell_V (\Gamma, x) \, \ell_{V'} (\Gamma', x).
  \end{align*}
  This result and (\ref{eq:defxi}) imply that $\xi_{V \cup V'} \,
(\Gamma \ast \Gamma', x) = \xi_V (\Gamma, x) \, \xi_{V'} (\Gamma',
x)$.
\end{proof}

\subsection{Cluster complexes and subdivisions} \label{subsec:cluster}
Let $\Phi$ be a finite root system of rank $n$. As in the introduction, we
will fix a positive system $\Phi^+$ with corresponding simple system $\Pi =
\{ \alpha_i: i \in I\}$, where $I$ is an $n$-element index set, and set
$\Phi_{\ge -1} := \Phi^+ \cup (-\Pi)$. For $J \subseteq I$, the standard
parabolic root subsystem $\Phi_J$ is endowed with the induced positive
system $\Phi^+_J = \Phi^+ \cap \Phi_J$ and corresponding simple system
$\Pi_J = \{\alpha_i: i \in J\}$.

The cluster complex $\Delta (\Phi)$ is a simplicial complex on the
vertex set $\Phi_{\ge -1}$. Its faces are the sets consisting of
mutually compatible elements of $\Phi_{\ge -1}$, where compatibility
is a symmetric binary relation on $\Phi_{\ge -1}$ defined in
\cite[Section~3]{FZ03}. We refer the reader to \cite{FZ03}
\cite[Section~4.3]{FR07} for the precise definition of compatibility
and collect the properties of $\Delta (\Phi)$ and its restriction
$\Delta_+ (\Phi)$ on the vertex set $\Phi^+$ which will be important
for us, in the following proposition. Part~(ii) is implicit in
\cite[Section~3]{FZ03} (see Lemma~3.12 and the proof of Theorem~1.10
there) and \cite[Section~8]{BW08}. The other parts follow directly
from the results of \cite[Section~3]{FZ03}.

\begin{proposition} \label{prop:cback}
\begin{enumerate}
\itemsep=0pt
\item[(i)]
The cluster complex $\Delta(\Phi)$ is homeomorphic to an
$(n-1)$-dimensional sphere.

\item[(ii)]
The complex $\Delta_+ (\Phi)$ is homeomorphic to an $(n-1)$-dimensional
ball.

\item[(iii)]
For $J \subseteq I$ we have $\link_{\Delta (\Phi)} (-\Pi_J) = \Delta
(\Phi_J)$.

\item[(iv)]
For $J \subseteq I$, the restriction of $\Delta (\Phi)$ to the vertex set
$(\Phi_J)_{\ge -1}$ is equal to $\Delta (\Phi_J)$ and that of $\Delta_+
(\Phi)$ to the vertex set $\Phi^+_J$ is equal to $\Delta_+ (\Phi_J)$.

\item[(v)]
If $\Phi$ is a direct product $\Phi_1 \times \Phi_2$, then $\Delta (\Phi)
= \Delta (\Phi_1) \ast \Delta (\Phi_2)$ and $\Delta_+ (\Phi) = \Delta_+
(\Phi_1) \ast \Delta_+ (\Phi_2)$.
\qed
\end{enumerate}
\end{proposition}

The following result of \cite{AT06} will be needed in
Section~\ref{sec:cproof} in order to compute the right-hand side
of (\ref{eq:clocalh}).

\begin{lemma} {\rm (\cite[Proposition~6.1]{AT06})}  \label{lem:h+}
For the $h$-polynomial of $\Delta_+ (\Phi)$ we have

  $$ h (\Delta_+ (\Phi), x) \ = \ \begin{cases}
    {\displaystyle \sum_{i=0}^n \, \frac{1}{i+1} \binom{n}{i} \binom{n-1}{i}
    x^i}, & \text{if \ $\xX = A_n$} \\ & \\
    {\displaystyle \sum_{i=0}^n \binom{n}{i} \binom{n-1}{i} x^i},
    & \text{if \ $\xX = B_n$} \\ & \\
    {\displaystyle \sum_{i=0}^n \left( \binom{n}{i} \binom{n-2}{i} +
    \binom{n-2}{i-2} \binom{n-1}{i} \right) x^i}, & \text{if \ $\xX = D_n$},
    \end{cases} $$
where $\xX$ is the Cartan-Killing type of $\Phi$.
\qed
\end{lemma}

We now formally define the cluster subdivision $\Gamma (\Phi)$. Given
a positive root $\alpha \in \Phi^+$, there is a unique set $J \subseteq I$
such that $\alpha$ is a positive linear combination of the elements of
$\Pi_J$. We call $\Pi_J$ the \emph{support} of $\alpha$ and for $E \in
\Delta_+ (\Phi)$, we denote by $\sigma(E)$ the union of the supports of the
elements of $E$. Equivalently, $\sigma(E)$ is the smallest set $\Pi_J
\subseteq \Pi$ such that $\alpha \in \Phi_J^+$ for every $\alpha \in E$.

  \begin{figure}
  \epsfysize = 2.5 in \centerline{\epsffile{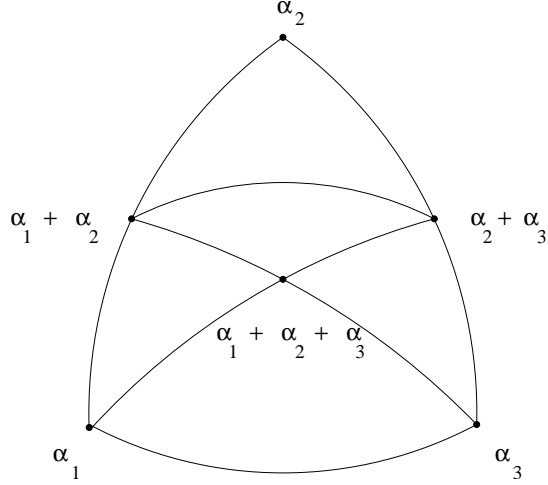}}
  \caption{The cluster subdivision of type $A_3$.}
  \label{fig:cposA3}
  \end{figure}

\begin{proposition} \label{prop:csub}
The map $\sigma: \Delta_+ (\Phi) \mapsto 2^\Pi$ defines a simplicial
subdivision $\Gamma(\Phi)$ of the simplex $2^\Pi$, whose local
$h$-polynomial is given by (\ref{eq:clocalh}).
\end{proposition}
\begin{proof}
It suffices to show that for every $J \subseteq I$: (a)
$\sigma^{-1} (2^{\Pi_J})$ is a subcomplex of $\Delta_+ (\Phi)$
which is homeomorphic to a ball of dimension $|J| - 1$; (b)
$\sigma^{-1} (\Pi_J)$ is the interior of this ball; and (c)
$\sigma^{-1} (2^{\Pi_J}) = \Delta_+ (\Phi_J)$. Indeed, (a) and (b)
confirm that $\sigma$ defines a simplicial subdivision of the
simplex $2^\Pi$ and (c) ensures that the restriction of this
subdivision to the face $\Pi_J$ of $2^\Pi$ is equal to $\Delta_+
(\Phi_J)$. Equation~(\ref{eq:clocalh}) is a consequence of the
last statement and Definition~\ref{def:localh}.

Part~(c) follows from the definition of the map $\sigma$ and
Proposition~\ref{prop:cback}~(iv) and part~(a) follows from (c) and
Proposition~\ref{prop:cback}~(ii). To verify (b), we may assume that
$J = I$. We need to show that the boundary of $\Delta_+ (\Phi)$ is
equal to the union of the subcomplexes $\Delta_+ (\Phi_J)$, where
$J$ runs through the proper subsets of $I$. For that, it suffices to
show that an $(n-2)$-dimensional face, say $E$, of $\Delta_+ (\Phi)$
is contained in a unique facet of $\Delta_+ (\Phi)$ if and only if
$E \in \Delta_+ (\Phi_J)$ for some $(n-1)$-element set $J \subseteq
I$. This is a consequence of parts~(i) and (iii) of
Proposition~\ref{prop:cback}. Indeed, part~(i) implies that $E$ is
contained in exactly two facets of $\Delta (\Phi)$. Part~(iii)
implies that at most one of these contains a negative simple root
and that this is the case if and only if $E \in \Delta_+ (\Phi_J)$
for some $(n-1)$-element set $J \subseteq I$.
\end{proof}

\begin{example} \label{ex:csubA3}
{\rm The complex $\Delta_+ (\Phi)$ and cluster subdivision
$\Gamma(\Phi)$ are drawn on Figure~\ref{fig:cposA3} for the root
system $\Phi$ of type $A_3$. The simple roots $\alpha_1, \alpha_2,
\alpha_3$ have been labeled so that $\alpha_1$ is orthogonal to
$\alpha_3$.

The subdivision $\Gamma(\Phi)$ triangulates the 2-dimensional simplex $2^\Pi$
into five 2-dimensional simplices, which are the facets of $\Delta_+
(\Phi)$. There is one interior vertex, namely $\alpha_1 + \alpha_2 + \alpha_3$.
The supports of $\alpha_1 + \alpha_2$ and $\alpha_2 + \alpha_3$ are equal to
$\{\alpha_1, \alpha_2\}$ and $\{\alpha_2, \alpha_3\}$, respectively. The
restriction of $\Gamma(\Phi)$ on the face $\{\alpha_1, \alpha_2\}$ of $2^\Pi$
is a subdivision of a 1-dimensional simplex with one interior vertex, namely
$\alpha_1 + \alpha_2$. }
\qed
\end{example}

\begin{remark} \label{rem:quiver}
{\rm One can define a cluster complex, and hence a corresponding
cluster subdivision, for every orientation of the Dynkin diagram
of $\Phi$ \cite{MRZ03}; see also \cite[Section~7]{Re07} (the
cluster complex of \cite{FZ03} \cite[Section~4.3]{FR07}, treated
here, corresponds to the alternating orientation). By
\cite[Proposition~3.4]{MRZ03} (see also
\cite[Proposition~7.3]{Re07}) and the results of
\cite[Section~6]{MRZ03}, the $h$-vector of the positive part of
the cluster complex and the local $h$-vector of the corresponding
cluster subdivision do not depend on the orientation chosen. \qed
}
\end{remark}

We conclude this section with the following lemma, which will be
used in the proof of Corollary~\ref{cor:clocalg}.

\begin{lemma} \label{lem:cjoin}
If $\Phi$ is a direct product $\Phi_1 \times \Phi_2$, then $\Gamma (\Phi)
= \Gamma (\Phi_1) \ast \Gamma (\Phi_2)$.
\end{lemma}
\begin{proof}
This statement follows from Proposition~\ref{prop:cback}~(v) and the
definitions of the cluster subdivision and the join of two
simplicial subdivisions.
\end{proof}

\subsection{Noncrossing partitions} \label{subsec:nc}
This section summarizes those concepts and results from the theory
of noncrossing partitions which are involved in the statements and
proofs of Theorems~\ref{thm:clocalh} and \ref{thm:clocalg}.

The set of noncrossing partitions of $\{1, 2,\dots,n\}$, which we
will denote by $\nc^A (n)$, was introduced and studied by Kreweras
\cite{Kr72}. It consists of all set partitions $\pi$ of $\{1,
2,\dots,n\}$ with the following property: if $a < b < c < d$ are
such that $a, c$ are contained in a block $B$ of $\pi$ and $b, d$
are contained in a block $B'$ of $\pi$, then $B=B'$. An example of
a noncrossing partition  for $n = 9$ is shown on
Figure~\ref{fig:nc9ex}. Among several other fundamental results,
Kreweras \cite[Section~4]{Kr72} showed that the cardinality of
$\nc^A (n)$ is equal to the $n$th Catalan number $\frac{1}{n+1}
\binom{2n}{n}$ and that

  \begin{equation} \label{eq:rankgenA}
    \sum_{\pi \in \nc^A (n)} x^{n - |\pi|} \ = \ \sum_{i=0}^n \,
    \frac{1}{i+1} \binom{n}{i} \binom{n-1}{i} x^i.
  \end{equation}

  \begin{figure}
  \epsfysize = 0.7 in \centerline{\epsffile{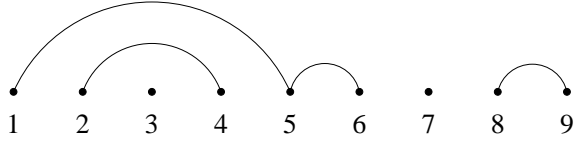}}
  \caption{The noncrossing partition $\{ \{1, 5, 6\}, \{2, 4\}, \{3\}, \{7\},
  \{8, 9\} \}$.}
  \label{fig:nc9ex}
  \end{figure}

We will say that a singleton block $\{b\}$ of $\pi \in \nc^A (n)$
is \emph{nested} if some block of $\pi$ contains elements $a$ and
$c$ such that $a < b < c$; otherwise we say that $\{b\}$ is
\emph{nonnested}. For the example of Figure~\ref{fig:nc9ex} the
singleton block $\{3\}$ is nested, while $\{7\}$ is not. Clearly,
a partition $\pi \in \nc^A (n)$ with nonnested singleton block
$\{b\}$ is determined by its restrictions to $\{1, 2,\dots,b-1\}$
and $\{b+1,\dots,n\}$, which are again noncrossing partitions.

Noncrossing partitions of type $B$ were defined by Reiner
\cite{Re97} as follows. A set partition $\pi$ of $\{1, 2,\dots,n\}
\cup \{-1, -2,\dots,-n\}$ is called a \textit{$B_n$-partition} if
the following conditions hold: (a) if $B$ is a block of $\pi$,
then $-B$ (the set obtained by negating the elements of $B$) is
also a block of $\pi$; and (b) there is at most one block of $\pi$
(called the \textit{zero block}, if present) which contains both
$i$ and $-i$ for some $i \in \{1, 2,\dots,n\}$. Such a partition
can be represented pictorially \cite[Section~2]{Ath98} by placing
the integers $1, 2,\dots,n, -1, -2,\dots,-n$ (in this order) along
a line and drawing arcs above the line between $i$ and $j$
whenever $i$ and $j$ lie in the same block $B$ of $\pi$ and no
other element between them does. The $B_n$-partition $\pi$ is
called \textit{noncrossing} if no two arcs in this diagram cross.
An example for $n = 7$ appears in Figure~\ref{fig:ncBex}. The set
of noncrossing $B_n$-partitions will be denoted by $\nc^B (n)$.

  \begin{figure}
  \epsfysize = 0.8 in \centerline{\epsffile{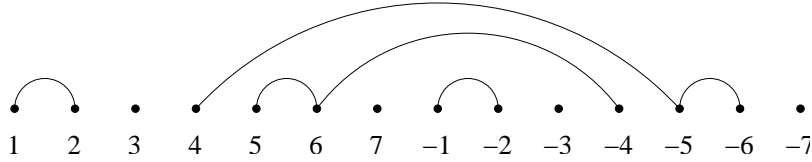}}
  \caption{A $B_7$-noncrossing partition.}
  \label{fig:ncBex}
  \end{figure}

We will be interested in the enumeration of noncrossing
$B_n$-partitions with no zero block, by the number of blocks.
Although we have not been able to locate the following statement
explicitly in the literature, its proof follows easily from that
of \cite[Theorem~2.3]{Ath98}.

\begin{lemma} \label{lem:ncB}
The number of partitions $\pi \in \nc^B (n)$ which have no zero
block and a total of $k$ pairs $\{B, -B\}$ of nonzero blocks is
equal to $\binom{n}{k} \binom{n-1}{k-1}$.
\end{lemma}
\begin{proof}
The proof of \cite[Theorem~2.3]{Ath98}, given in
\cite[Section~4]{Ath98}, shows that the partitions $\pi \in \nc^B
(n)$ which have no zero block and a total of $k$ pairs $\{B, -B\}$
of nonzero blocks are in one-to-one correspondence with pairs $(S,
f)$, where $S$ is a $k$-element subset of $\{1, 2,\dots,n\}$ and
$f: S \to \{1, 2,\dots\}$ is a map whose values sum to $n$. Since
there are $\binom{n}{k}$ ways to choose $S$ and, for any such
choice, there are $\binom{n-1}{k-1}$ ways to choose $f$, the
result follows.
\end{proof}

\medskip
We will say that a singleton block $\{b\}$ of $\pi \in \nc^B (n)$ is \emph{nested}
if some block of $\pi$ contains an element which precedes $b$ and one which succeeds
$b$ in the linear ordering $1, 2,\dots,n, -1, -2,\dots,-n$; otherwise we say that
$\{b\}$ is \emph{nonnested}. The example of Figure~\ref{fig:ncBex} has the
nonnested positive singleton block $\{3\}$ and the nested positive singleton block
$\{7\}$. A partition $\pi \in \nc^B (n)$ with nonnested positive singleton block
$\{b\}$ is determined by its restrictions to $\{1, 2,\dots,b-1\}$ and
$\{b+1,\dots,n\} \cup \{-b-1,\dots,-n\}$, which are noncrossing partitions of
types $A$ and $B$, respectively.

\section{Proofs for cluster subdivisions}
\label{sec:cproof}

This section provides proofs for Theorems~\ref{thm:clocalh} and
\ref{thm:clocalg} and Corollary~\ref{cor:clocalg}. As part of the
proof for the types $A_n$ and $B_n$, combinatorial interpretations
similar to those of Theorem~\ref{thm:clocalh} for the numbers
$\ell_i (\Phi)$ are provided for the numbers $\xi_i (\Phi)$.

As in previous sections, $\Phi = \Phi_I$ will be a finite root system of rank $n$.
We will denote by $\dD(\Phi)$ the Dynkin diagram of $\Phi$ and identify the vertex
set of $\dD(\Phi)$ with the $n$-element index set $I$. We will first treat the root
systems of types $A_n$, $B_n$ and $D_n$.

\subsection{The root system $A_n$} \label{subsec:A}
The following proposition is the main result of this section. Note
that noncrossing partitions with no singleton block and given
number of blocks, which appear there, were considered and
enumerated by Kreweras \cite[Section~5]{Kr72}.

\begin{proposition} \label{prop:An}
For the root system $\Phi$ of type $A_n$ the following hold:
  \begin{itemize}
    \item[$\bullet$] $\ell_i (\Phi)$ is equal to the number of partitions $\pi \in
                     \nc^A(n)$ with $i$ blocks, such that every singleton block of
                     $\pi$ is nested,
    \item[$\bullet$] $\xi_i (\Phi)$ is equal to the number of partitions $\pi \in
                     \nc^A(n)$ which have no singleton block and a total of $i$
                     blocks.
  \end{itemize}
Moreover, we have the explicit formula
  \begin{equation} \label{eq:xiA}
    \xi_i (\Phi) \ = \ \begin{cases}
    0, & \text{if \ $i=0$} \\
    & \\
    \displaystyle \frac{1}{n-i+1} \binom{n}{i} \binom{n-i-1}{i-1}, &
    \text{if \ $1 \le i \le \lfloor n/2 \rfloor$}. \end{cases}
  \end{equation}
\end{proposition}

\medskip
For the first few values of $n$ we have

 $$ \sum_{i=0}^n \, \ell_i (\Phi) x^i \ = \ \begin{cases}
    0, & \text{if \ $n = 1$} \\
    x, & \text{if \ $n = 2$} \\
    x + x^2, & \text{if \ $n = 3$} \\
    x + 4x^2 + x^3, & \text{if \ $n = 4$} \\
    x + 8x^2 + 8x^3 + x^4, & \text{if \ $n = 5$} \\
    x + 13x^2 + 29x^3 + 13x^4 + x^5, & \text{if \ $n = 6$} \\
    x + 19x^2 + 73x^3 + 73x^4 + 19x^5 + x^6, & \text{if \ $n = 7$} \\
    x + 26x^2 + 151x^3 + 266x^4 + 151x^5 + 26x^6 + x^7, & \text{if \ $n = 8$}
    \end{cases} $$
and

 $$ \sum_{i=0}^{\lfloor n/2 \rfloor} \, \xi_i (\Phi) x^i \ = \ \begin{cases}
    0, & \text{if \ $n=1$} \\
    x, & \text{if \ $n = 2, 3$} \\
    x + 2x^2, & \text{if \ $n = 4$} \\
    x + 5x^2, & \text{if \ $n = 5$} \\
    x + 9x^2+ 5x^3, & \text{if \ $n = 6$} \\
    x + 14x^2 + 21x^3, & \text{if \ $n = 7$} \\
    x + 20x^2 + 56x^3 + 14x^4, & \text{if \ $n = 8$}.  \end{cases} $$

\medskip
The Dynkin diagram $\dD(\Phi)$ is a path on the vertex set $I$. For notational
convenience we set $I = \{1, 2,\dots,n\}$, where $i$ and $i+1$ are adjacent in
$\dD(\Phi)$ for $1 \le i \le n-1$.

\medskip
\noindent
\begin{proof}[Proof of Proposition~\ref{prop:An}] We need to compute the right-hand
side of (\ref{eq:clocalh}), so we focus on $h (\Delta_+ (\Phi_J),
x)$. Lemma~\ref{lem:h+} and Equation~(\ref{eq:rankgenA}) show that
  \begin{equation} \label{eq:AA}
    h (\Delta_+ (\Phi_I), x) \ = \ \sum_{\pi \in \nc^A (n)} x^{n - |\pi|}.
  \end{equation}
For general $J \subseteq I$ we have a direct product decomposition
$\Phi_J = \Phi_1 \times \cdots \times \Phi_k$ into irreducible
subsystems $\Phi_1,\dots,\Phi_k$. The Dynkin diagrams of
$\Phi_1,\dots,\Phi_k$ are the connected components of the diagram
obtained from $\dD (\Phi)$ by deleting the vertices in $I \sm J$.
Since $\dD (\Phi)$ is a path with no multiple edges, each $\Phi_i$
is again a root system of type $A$. Denoting by $p_i$ the rank of
$\Phi_i$ and using Proposition~\ref{prop:cback} (v) and
Equation~(\ref{eq:AA}) we find that
  \begin{align*}
    h (\Delta_+ (\Phi_J), x) & \ = \ h (\Delta_+ (\Phi_1) \ast \cdots \ast \Delta_+
    (\Phi_k), x) \ = \ \prod_{i=1}^k \ h (\Delta_+ (\Phi_i), x) \\
    & \ = \ \prod_{i=1}^k \ \sum_{\pi \in \nc^A (p_i)} x^{p_i - |\pi|} \ = \
    \sum_{\pi \in \nc^A (J)} x^{n - |\pi|},
  \end{align*}
where $\nc^A (J)$ denotes the set of partitions $\pi \in \nc^A (n)$ such that
$\{a\}$ is a nonnested singleton of $\pi$ for every $a \in I \sm J$. The previous
computation and (\ref{eq:clocalh}) imply that
  \begin{equation} \label{eq:1An}
    \sum_{i=0}^n \, \ell_i (\Phi) x^i \ = \ \sum_{J \subseteq I} \ (-1)^{|I \sm J|}
    \sum_{\pi \in \nc^A (J)} x^{n - |\pi|}.
  \end{equation}
A simple application of the principle of inclusion-exclusion shows that the
right-hand side of (\ref{eq:1An}) is equal to the sum of $x^{n - |\pi|}$, where
$\pi$ runs through those partitions in $\nc^A(n)$ which have no nonnested singleton
block. This result and the fact that $\ell_i (\Phi) = \ell_{n-i} (\Phi)$ yield the
desired interpretation for $\ell_i (\Phi)$.

To prove the interpretation claimed for $\xi_i (\Phi)$ we need to show that
  \begin{equation} \label{eq:2An}
    \sum_{i=0}^n \, \ell_i (\Phi) x^i \ = \ \sum_{i=0}^{\lfloor n/2 \rfloor}
    \, m_i \, x^i (1+x)^{n-2i},
  \end{equation}
where $m_i$ is the number of partitions $\pi \in \nc^A(n)$ with a
total of $i$ blocks, none of which is a singleton. Let us denote
by $\nc_0^A (n)$ the subset of $\nc^A (n)$ consisting of those
noncrossing partitions, every singleton block of which is nested.
We define an equivalence relation on $\nc^A (n)$ by declaring two
partitions $\pi_1$ and $\pi_2$ equivalent if there is a one-to-one
correspondence, say $f$, from the set of nonsingleton blocks of
$\pi_1$ to the set of nonsingleton blocks of $\pi_2$ such that for
every nonsingleton block $B$ of $\pi_1$ the sets $B$ and $f(B)$
have the same minimum and the same maximum element. For example,
the partition in Figure~\ref{fig:nc9ex} is equivalent to a total
of four noncrossing partitions, namely itself, $\{ \{1, 5, 6\},
\{2, 3, 4\}, \{7\}, \{8, 9\} \}$, $\{ \{1, 6\}, \{2, 3, 4\},
\{5\}, \{7\}, \{8, 9\} \}$ and $\{ \{1, 6\}, \{2, 4\}, \{3\},
\{5\}, \{7\}, \{8, 9\} \}$.

We leave it to the reader to check that this relation restricts to an equivalence
relation on $\nc_0^A (n)$ and that each equivalence class within $\nc_0^A (n)$
contains a unique partition $\pi_0$ having no singleton block. Moreover, for the
equivalence class $O(\pi_0)$ of such a partition $\pi_0 \in \nc_0^A (n)$ we have
  $$ \sum_{\pi \in O(\pi_0)} x^{|\pi|} \ = \ x^i (1+x)^{n-2i}, $$
where $i$ is the number of blocks of $\pi_0$. Summing the previous equation over
all elements $\pi_0 \in \nc_0^A (n)$ which have no singleton block we get
(\ref{eq:2An}).

Finally, (\ref{eq:xiA}) is a consequence of the equality $\xi_i (\Phi) = m_i$ and
the results of \cite[p.~344]{Kr72}, which enumerate noncrossing partitions with
no singleton block and given number of blocks.
\end{proof}

\subsection{The root system $B_n$} \label{subsec:B}
This section proves the following statement on the case $\xX = B_n$.

\begin{proposition} \label{prop:Bn}
For the root system $\Phi$ of type $B_n$ the following hold:
  \begin{itemize}
    \item[$\bullet$] $\ell_i (\Phi)$ is equal to the number of partitions $\pi \in
                     \nc^B(n)$ with no zero block and $i$ pairs $\{B, -B\}$
                     of nonzero blocks, such that every positive singleton block
                     of $\pi$ is nested,
    \item[$\bullet$] $\xi_i (\Phi)$ is equal to the number of partitions $\pi \in
                     \nc^B (n)$ which have no zero block, no singleton block
                     and a total of $i$ pairs $\{B, -B\}$ of nonzero blocks.
  \end{itemize}
Moreover, we have the explicit formula
  \begin{equation} \label{eq:xiB}
    \xi_i (\Phi) \ = \ \begin{cases}
    0, & \text{if \ $i=0$} \\
    & \\
    \displaystyle \binom{n}{i} \binom{n-i-1}{i-1}, &
    \text{if \ $1 \le i \le \lfloor n/2 \rfloor$}. \end{cases}
  \end{equation}
\end{proposition}

\medskip
For the first few values of $n$ we have

 $$ \sum_{i=0}^n \, \ell_i (\Phi) x^i \ = \ \begin{cases}
    2x, & \text{if \ $n = 2$} \\
    3x + 3x^2, & \text{if \ $n = 3$} \\
    4x + 14x^2 + 4x^3, & \text{if \ $n = 4$} \\
    5x + 35x^2 + 35x^3 + 5x^4, & \text{if \ $n = 5$} \\
    6x + 69x^2 + 146x^3 + 69x^4 + 6x^5, & \text{if \ $n = 6$} \\
    7x + 119x^2 + 427x^3 + 427x^4 + 119x^5 + 7x^6, & \text{if \ $n = 7$} \\
    \end{cases} $$
and

 $$ \sum_{i=0}^{\lfloor n/2 \rfloor} \, \xi_i (\Phi) x^i \ = \ \begin{cases}
    2x, & \text{if \ $n = 2$} \\
    3x, & \text{if \ $n = 3$} \\
    4x + 6x^2, & \text{if \ $n = 4$} \\
    5x + 20x^2, & \text{if \ $n = 5$} \\
    6x + 45x^2+ 20x^3, & \text{if \ $n = 6$} \\
    7x + 84x^2 + 105x^3, & \text{if \ $n = 7$} \\
    8x + 140x^2 + 336x^3 + 70x^4, & \text{if \ $n = 8$}.  \end{cases} $$

\medskip
The Dynkin diagram $\dD(\Phi)$ is a path on the vertex set $I = \{1,
2,\dots,n\}$ with one double edge. We will assume that $i$ and $i+1$ are
adjacent in $\dD(\Phi)$ for $1 \le i \le n-1$ and that the double edge
connects vertices $n-1$ and $n$.

\medskip
\noindent
\begin{proof}[Proof of Proposition~\ref{prop:Bn}] A proof which parallels
that of Proposition~\ref{prop:An} can be given as follows. We
denote by $\nc_+^B (n)$ the set of partitions $\pi \in \nc^B (n)$
which do not have a zero block. To compute the right-hand side of
(\ref{eq:clocalh}), we consider $h (\Delta_+ (\Phi_J), x)$ for $J
\subseteq I$. Lemmas~\ref{lem:h+} and \ref{lem:ncB}, together with
some straightforward computations, show that
  \begin{equation} \label{eq:BB}
    h (\Delta_+ (\Phi_I), x) \ = \ \sum_{\pi \in \nc_+^B (n)} x^{n - \|\pi\|},
  \end{equation}
where $\|\pi\|$ stands for the number of pairs $\{B, -B\}$ of (nonzero) blocks
of $\pi$. For general $J \subseteq I$ we claim that
 \begin{equation} \label{eq:nc+B}
   h (\Delta_+ (\Phi_J), x) \ = \ \sum_{\pi \in \nc_+^B (J)} x^{n - \|\pi\|},
 \end{equation}
where $\nc_+^B (J)$ denotes the set of partitions $\pi \in \nc_+^B (n)$ such
that $\{a\}$ is a nonnested (positive) singleton block of $\pi$ for every $a
\in I \sm J$. Given (\ref{eq:nc+B}), the first statement follows by an
application of inclusion-exclusion, as in the type $A_n$ case.

The proof of (\ref{eq:nc+B}) proceeds without essential change if
$n-1$ or $n$ does not belong to $J$. Otherwise we have $\{n-1, n\}
\subseteq J$ and the argument in the proof of
Proposition~\ref{prop:An} should be modified as follows. Let $b$
denote the maximum element of $I \sm J$. Then $\{b+1,\dots,n\}$ is
the vertex set of the Dynkin diagram of one of the irreducible
components, say $\Phi_k$, of $\Phi_J$. This component is of type
$B$, while each of $\Phi_1,\dots,\Phi_{k-1}$ is of type $A$.
Moreover, given $\pi \in \nc_+^B (J)$, the restriction of $\pi$ on
$\{b+1,\dots,n\} \cup \{-b-1,\dots,-n\}$ is a noncrossing
partition of type $B$, while that on the vertex set of the Dynkin
diagram of each of $\Phi_1,\dots,\Phi_{k-1}$ is a noncrossing
partition of type $A$. Thus (\ref{eq:nc+B}) follows by the
computation in proof of Proposition~\ref{prop:An} and the use of
(\ref{eq:AA}) and (\ref{eq:BB}).

For the second statement, we need to replace the equivalence
relation on $\nc^A (n)$ by one on $\nc_+^B (n)$, defined as
follows. Suppose that $\pi \in \nc_+^B (n)$ has a nested positive
singleton block $\{b\}$. Then there is a unique block $B \in \pi$
such that replacing the blocks $B, -B, \{b\}$ and $\{-b\}$ of
$\pi$ by the unions $B \cup \{b\}$ and $(-B) \cup \{-b\}$ results
in a noncrossing partition $\pi' \in \nc_+^B (n)$. The required
equivalence relation on $\nc_+^B (n)$ is defined as the finest
equivalence relation under which $\pi$ and $\pi'$ are equivalent
for all such pairs $(\pi, b)$. For example, the partition in
Figure~\ref{fig:ncBex} is equivalent to exactly one other
noncrossing partition, of which $\{5, 6, 7, -4\}$ is a block. The
proof then proceeds as in the type $A_n$ case with only trivial
adjustments; the details are left to the reader.

Finally, to deduce the explicit formula (\ref{eq:xiB}) we argue as
in the proof of Lemma~\ref{lem:ncB}. The proof of
\cite[Theorem~2.3]{Ath98} shows that the partitions $\pi \in \nc^B
(n)$ which have no zero block, no singleton block and a total of
$i$ pairs $\{B, -B\}$ of nonzero blocks are in one-to-one
correspondence with pairs $(S, f)$, where $S$ is an $i$-element
subset of $\{1, 2,\dots,n\}$ and $f: S \to \{2, 3,\dots\}$ is a
function whose values sum to $n$. Clearly, the number of such
pairs is given by the right-hand side of (\ref{eq:xiB}) and the
proof follows.
\end{proof}

\subsection{The root system $D_n$} \label{subsec:D}
This section proves the following part of
Theorems~\ref{thm:clocalh} and \ref{thm:clocalg}.

\begin{proposition} \label{prop:Dn}
For the root system $\Phi$ of type $D_n$ we have:

  \begin{align*}
    \ell_i (\Phi) & \ = \ (n-2) \cdot \, \# \, \{\pi \in \nc^A (n-1): |\pi| = i\} \\
    &  \\
    & \ = \ \begin{cases}
    0, & \text{if \ $i=0$} \\
    & \\
    \displaystyle \frac{n-2}{i} \binom{n-1}{i-1} \binom{n-2}{i-1}, &
    \text{if \ $1 \le i \le n$}, \end{cases}
  \end{align*}
and

  $$ \xi_i (\Phi) \ = \ \begin{cases}
    0, & \text{if \ $i=0$} \\
    & \\
    \displaystyle \frac{n-2}{i} \binom{2i-2}{i-1} \binom{n-2}{2i-2}, &
    \text{if \ $1 \le i \le \lfloor n/2 \rfloor$}. \end{cases} $$
\end{proposition}

\bigskip
For the first few values of $n$ we have

 $$ \sum_{i=0}^n \, \ell_i (\Phi) x^i \ = \ \begin{cases}
    2x + 6x^2 + 2x^3, & \text{if \ $n = 4$} \\
    3x + 18x^2 + 18x^3 + 3x^4, & \text{if \ $n = 5$} \\
    4x + 40x^2 + 80x^3 + 40x^4 + 4x^5, & \text{if \ $n = 6$} \\
    5x + 75x^2 + 250x^3 + 250x^4 + 75x^5 + 5x^6, & \text{if \ $n = 7$}
    \end{cases} $$
and

 $$ \sum_{i=0}^{\lfloor n/2 \rfloor} \, \xi_i (\Phi) x^i \ = \ \begin{cases}
    2x + 2x^2, & \text{if \ $n = 4$} \\
    3x + 9x^2, & \text{if \ $n = 5$} \\
    4x + 24x^2 + 8x^3, & \text{if \ $n = 6$} \\
    5x + 50x^2 + 50x^3, & \text{if \ $n = 7$} \\
    6x + 90x^2 + 180x^3 + 30x^4, & \text{if \ $n = 8$}.  \end{cases} $$

\bigskip
One can easily deduce from Proposition~\ref{prop:Dn} a combinatorial interpretation
to the numbers $\xi_i (\Phi)$; see also \cite[Section~11.3]{PRW08}. We are not aware,
however, of one which is analogous to those in Propositions~\ref{prop:An} and
\ref{prop:Bn} for types $A_n$ and $B_n$.

The following notation and enumerative result will be used in the
proof of Proposition~\ref{prop:Dn}. We will write
   $$ C_n (x)  \ := \ \sum_{\pi \in \nc^A (n)} x^{|\pi| - 1} \ = \ \sum_{\pi \in
    \nc^A (n)} x^{n - |\pi|}
    \ = \ \sum_{i=0}^n \, \frac{1}{i+1} \binom{n}{i} \binom{n-1}{i}
    x^i $$
and
  \begin{equation} \label{eq:Fdef}
    F (x, t) \ := \ \sum_{n \ge 1} \, C_n (x) \, t^n \ = \ t + (1 + x) \,
    t^2 + (1 + 3x + x^2) \, t^3 + \cdots
  \end{equation}

\noindent Then (see, for instance, \cite[Equation~(11)]{PRW08} and
\cite[Exercise~6.36]{StaEC2}) we have
  \begin{equation} \label{eq:F}
    F(x, t) \ = \ xt \, F^2 (x, t) + (1 + x)t \, F(x, t) + t.
  \end{equation}

We will label the vertices of the Dynkin diagram $\dD(\Phi)$ so that $i$ and
$i+1$ are adjacent in $\dD(\Phi)$ for $1 \le i \le n-3$, while $n-2$ is adjacent
to both $n-1$ and $n$.

\medskip
\noindent
\begin{proof}[Proof of Proposition~\ref{prop:Dn}]
Let us write $\ell_n (x) := \ell_I (\Gamma (\Phi), x) = \sum_{i=0}^n \ell_i
(\Phi) x^i$ for $n \ge 4$. The proposed formula for $\ell_i (\Phi)$ is equivalent
to the equation
  \begin{equation} \label{eq:localhAD}
   \ell_n (x) \ = \ (n-2) \cdot x \, C_{n-1} (x).
  \end{equation}
The formula for $\xi_i (\Phi)$ follows from that and the known
explicit formula (see \cite[Proposition~11.14]{PRW08}) for the
$\gamma$-polynomial associated to $C_n (x)$. Thus, it suffices to
prove (\ref{eq:localhAD}).

We begin by rewriting the right-hand side of (\ref{eq:clocalh}) in
the following way. For $1 \le r \le n$, we will denote by $\jJ_r$
the collection of all subsets $J \subseteq I$ which contain $\{1,
2,\dots,r-1\}$ but do not contain $r$. Using
Proposition~\ref{prop:cback} (v) and the type $A_n$ case of
Lemma~\ref{lem:h+}, we find that
 $$ \sum_{J \in \jJ_r} \ (-1)^{|I \sm J|} \ h (\Delta_+ (\Phi_J), x) \ = \
    \begin{cases}
    - \, \ell_{n-1} (x), & \text{if \ $r = 1$} \\
    - \, C_{r-1}(x) \, \ell_{n-r} (x), & \text{if \ $2 \le r \le n-3$} \\
    0, & \text{if \ $r = n-2$}, \\
    C_{n-2}(x) - C_{n-1}(x), & \text{if \ $r = n-1$}, \\
    - \, C_{n-1}(x), & \text{if \ $r = n$}.
    \end{cases} $$
As a result, (\ref{eq:clocalh}) can be rewritten as
  $$\ell_n (x) \ = \ h (\Delta_+ (\Phi_I), x) \, - \, \ell_{n-1} (x) \, - \,
    \sum_{r=2}^{n-3} C_{r-1} (x) \, \ell_{n-r} (x) \, + \, C_{n-2} (x) \, - \, 2 C_{n-1}
    (x).$$
Thus, using induction on $n$, it suffices to prove that
  \begin{align} \label{eq:localhDC}
    h (\Delta_+ (\Phi_I), x) &= (n-2)x C_{n-1} (x) \, + \, (n-3)x C_{n-2} (x)
    \nonumber \\
    & \hspace{10pt} + \ \sum_{r=2}^{n-3} \, (n-r-2) x C_{r-1} (x) C_{n-r-1} (x) \, - \,
    C_{n-2} (x) \ + \ 2 C_{n-1} (x)
  \end{align}
for $n \ge 4$. Let $R_n (x)$ denote the right-hand side of (\ref{eq:localhDC})
and $S_n(x)$ denote the sum which appears there. It follows directly from
(\ref{eq:Fdef}) that
  $$ \sum_{n\ge 4} S_n (x) t^n \ = \ xt^3 F(x, t) \, \frac{\partial F}{\partial t}
     (x, t) \, - \, xt^2 F^2 (x, t). $$
Using (\ref{eq:F}), as well as the equation which results from that by
differentiation with respect to $t$, we can rewrite the previous equation as
  $$ \sum_{n \ge 4} 2 S_n (x) t^n \ = \ 2 (1+x) t^2 F(x,t) \, + \, 2 t^2 \, - \,
     3 t F(x, t) \, + \, (t^2 - t^3 - xt^3) \, \frac{\partial F}{\partial t}
     (x, t). $$
Equating the coefficients of $t^n$ in the two sides above, we conclude that
  $$ 2 S_n (x) \ = \ (n - 4) C_{n-1} (x) \, - \, (n - 4) (1+x) C_{n-2} (x) $$
and hence that
  $$ R_n (x) \ = \ (n-2) x C_{n-1} (x) \, + \, \frac{n}{2} C_{n-1} (x) \, + \,
     (\frac{n}{2} - 1) (x-1) C_{n-2} (x). $$
Equation~(\ref{eq:localhDC}) follows from the formula for $h
(\Delta_+ (\Phi_I), x)$, given by the type $D_n$ case of
Lemma~\ref{lem:h+}, and the previous expression for $R_n(x)$ by
straightforward computation. This completes the proof of the
proposition.
\end{proof}

\medskip
\noindent
\begin{proof}[Proof of Theorems~\ref{thm:clocalh} and \ref{thm:clocalg}]
The cases $\xX \in \{A_n, B_n, D_n\}$ are covered by
Propositions~\ref{prop:An}, \ref{prop:Bn} and \ref{prop:Dn}. For
$\xX \in \{F_4, E_6, E_7, E_8\}$ the proposed formulas follow from
(\ref{eq:clocalh}) by explicit computation, based on the formulas
for $h (\Delta_+ (\Phi), x)$ given in \cite[Section~6]{AT06}. It
remains to comment on the cases of types $I_2 (m), H_3$ and $H_4$.

For types $I_2 (m)$ and $H_3$, it follows from the theory of local
$h$-vectors (see parts (c) and (d) of \cite[Example~2.3]{Sta92})
that $\xi_I (\Gamma (\Phi), x) = tx$, where $t$ is the number of
interior vertices of $\Gamma (\Phi)$. We have $t = m-2$ for $\xX =
I_2 (m)$ and $t = 8$ for $\xX = H_3$ (see \cite[Figure~1]{BW08} or
Remark~1 in Section~\ref{sec:rem}) and the proposed formulas
follow. Finally, let $\xX = H_4$. From (\ref{eq:clocalh}),
(\ref{eq:defclocalg}) and the fact that $\xi_0 (\Phi) = \ell_0
(\Phi) = 0$ we get
  $$ \xi_1 (\Phi) x(1+x)^2 + \xi_2 (\Phi) x^2 \ = \ \sum_{J \subseteq I} \
     (-1)^{|I \sm J|} \ h (\Delta_+ (\Phi_J), x). $$
Setting $x=1$ in the previous equality we get
  \begin{equation} \label{eq:xiH4}
    4 \xi_1 (\Phi) + \xi_2 (\Phi) \ = \ \sum_{J \subseteq I} \ (-1)^{|I \sm J|}
    \ N^+ (\Phi_J),
  \end{equation}
where $N^+ (\Psi)$ denotes the number of facets of $\Delta_+
(\Psi)$ (i.e., the number of positive clusters for the root system
$\Psi$). The right-hand side of (\ref{eq:xiH4}) can be easily
computed by hand (it equals 208), using
\cite[Proposition~3.9]{FZ03} and Proposition~\ref{prop:cback} (v).
Since $\xi_1 (\Phi) = \ell_1 (\Phi)$ is equal to the number of
interior vertices of $\Gamma (\Phi)$, we have $\xi_1 (\Phi) = 42$
(see Remark~1 in Section~\ref{sec:rem}). It follows from
(\ref{eq:xiH4}) that $\xi_2 (\Phi) = 40$.
\end{proof}

\medskip
\noindent
\begin{proof}[Proof of Corollary~\ref{cor:clocalg}] Theorem~\ref{thm:clocalg} shows
that the statement holds when $\Phi$ is irreducible. The general
case then follows from Lemmas~\ref{lem:localjoin} and
\ref{lem:cjoin}.
\end{proof}

\section{Proof of Theorem~\ref{thm:bary}}
\label{sec:bproof}

We first review two of the tools from the combinatorics of
permutations which will be used in the proof of
Theorem~\ref{thm:bary}. Throughout this section, we will denote by
$\eE_n$ the set of permutations in $\sS_n$ for which every left to
right maximum is a descent.

\medskip
\noindent
\textbf{Descents and excedances.}
Given a permutation $w \in \sS_n$, we may write $w$ in cycle form so that each cycle
begins with its largest element and the cycles of $w$ are arranged in the increasing
order of their largest elements (this is the standard representation of $w$, discussed
on \cite[p.~17]{StaEC1}). We denote by $\phi(w)$ the sequence (or word) which is
obtained after removing the parentheses from the cycles of $w$, considered as a
permutation in $\sS_n$. For instance, if $n=9$ and $w = (5 \, 2 \, 4)(6 \, 1)(8) (9
\, 7 \, 3)$ in standard cycle form, then $\phi(w) = (5, 2, 4, 6, 1, 8, 9, 7, 3)$ is
the permutation in $\sS_9$ which maps 1 to 5, 2 to itself, 3 to 4 etc. The following
properties hold (recall that $\dD_n$ denotes the set of derangements in $\sS_n$):
  \begin{itemize}
    \item[(a)] the map $\phi: \sS_n \to \sS_n$ is bijective,
    \item[(b)] $\phi (\dD_n) = \eE_n$,
    \item[(c)] for $w \in \sS_n$ and $1 \le i \le n$ we have $w(i) < i$ if and only
    if $i$ is a descent of $\phi(w)$.
  \end{itemize}
We will denote by $\hat{\phi}: \dD_n \to \eE_n$ the bijective map induced by $\phi$
on the set $\dD_n$.

  \medskip
  \noindent
  \textbf{The Foata-Sch\"utzenberger-Strehl action.}
  We will need the following variant of the Foata-Sch\"utzenberger-Strehl action on
  permutations; see, for instance, \cite[Section~V.1]{FSc70} \cite{Fo72, FS74}. Up to
  date expositions and several applications of this construction can be found in
  \cite{Bra08, PRW08}.

  We let $w = (w_1, w_2,\dots, w_n)$ be a permutation in $\eE_n$, where $w_i = w(i)$
  for $1 \le i \le n$, and set $w_0 = 0$ and $w_{n+1} = n+1$. A \emph{double ascent} of
  $w$ is an index $1 \le i \le n$ such that $w_{i-1} < w_i < w_{i+1}$. Given a double
  ascent or a double descent $i$ of $w$, we define the permutation $\psi_i (w) \in
  \sS_n$ as follows: If $i$ is a double ascent of $w$, then $\psi_i (w)$ is the
  permutation obtained from $w$ by moving $w_i$ between $w_j$ and $w_{j+1}$, where $j$
  is the largest index satisfying $1 \le j < i$ and $w_j > w_i > w_{j+1}$ (note that
  such an index exists, since $w \in \eE_n$ and hence $i$ is not a left to right
  maximum of $w$). Similarly, if $i$ is a double descent of $w$, then $\psi_i (w)$ is
  the permutation obtained from $w$ by moving $w_i$ between $w_j$ and $w_{j+1}$, where
  $j$ is the smallest index satisfying $i < j \le n$ and $w_j < w_i < w_{j+1}$ (note
  that such an index exists, since $w_{n+1} = n+1$). For instance, for the example of
  Figure~\ref{fig:perm} we have $\psi_4 (w) = (7, 5, 3, 1, 6, 9, 8, 2, 4)$ and $\psi_7
  (w) = (7, 3, 1, 5, 6, 9, 2, 4, 8)$. Since the values at left to right maxima are
  unchanged when passing from $w$ to $\psi_i (w)$, we have $\psi_i (w) \in \eE_n$ in
  both cases.

  \begin{figure}
  \epsfysize = 2.8 in \centerline{\epsffile{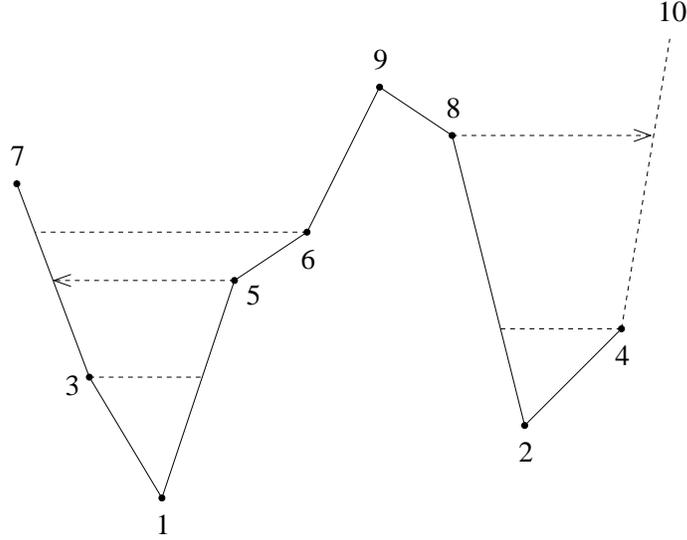}}
  \caption{The permutation $w = (7, 3, 1, 5, 6, 9, 8, 2, 4) \in \eE_9$.}
  \label{fig:perm}
  \end{figure}

  We call two permutations in $\eE_n$ equivalent (under the
  Foata-Sch\"utzenberger-Strehl action on
  $\eE_n$) if one can be obtained by applying a sequence of maps of the form $\psi_i$
  to the other. We leave it to the reader to check that this defines an equivalence
  relation on $\eE_n$ and that each equivalence class contains a unique element
  having no double descent. Moreover, if $w \in \eE_n$ has no double descent and $k$
  double ascents, then the equivalence class $O(w)$ of $w$ has $2^k$ elements and
  exactly $\binom{k}{j}$ of them have $j$ descents more than $w$, so that
    \begin{equation} \label{eq:orbits}
      \sum_{u \in O(w)} \, x^{\des(u)} \ = \ x^{\des(w)} (1+x)^k \ = \ x^{\des(w)}
      (1+x)^{n - 2\des(w)}.
    \end{equation}

  \noindent
    \begin{proof}[Proof of Theorem~\ref{thm:bary}] Starting from
    (\ref{eq:localbary}) we find that
      $$\ell_V (\Gamma, x)  \ = \ \sum_{u \in \dD_n} \, x^{\ex(u)} \ = \ \sum_{u \in
        \dD_n} \, x^{\ex(u^{-1})} \ = \ \sum_{u \in \dD_n} \, x^{n-\ex(u)}
         \ = \ \sum_{u \in \eE_n} \, x^{\des(u)},$$
    where the last equality uses property (c) for the map $\hat{\phi}: \dD_n \to
    \eE_n$. Summing (\ref{eq:orbits}) over all equivalence classes of the
    Foata-Sch\"utzenberger-Strehl action on $\eE_n$ we get
      $$ \sum_{u \in \eE_n} \, x^{\des(u)} \ = \ \sum_{w \in \hat{\eE}_n} \,
         x^{\des(w)} (1+x)^{n - 2\des(w)}, $$
    where $\hat{\eE}_n$ denotes the set of permutations $w \in \eE_n$ with no double
    descent. From the previous equalities and (\ref{eq:defxi}) we conclude that $\xi_i$
    is equal to the number of permutations $w \in \hat{\eE}_n$ with $\des(w) = i$, so
    we have derived interpretation (iii) in the theorem. The latter and property (c),
    applied to the map $\hat{\phi}: \dD_n \to \eE_n$, imply that $\xi_i$ is also equal
    to the number of derangements $w \in \dD_n$ with $n-i$ excedances and no index $j$
    satisfying $w(j) < j < w^{-1} (j)$. Passing to the inverse permutation $w^{-1}$
    leads to interpretation (ii) of the theorem.

    Finally, to check the equality between (i) and (ii), we work with descending
    (instead of ascending) runs. We observe that the map $\hat{\phi}: \dD_n \to \eE_n$
    induces a bijection from the set of derangements $w \in \dD_n$ with no double
    excedance onto the set of permutations in $\sS_n$ with no descending run of length
    one. Moreover, the number of excedances of such $w$ is equal to the number of
    descending runs of $\hat{\phi} (w)$ and the proof follows.
    \end{proof}

\section{Remarks}
\label{sec:rem}

1. It follows from the results of \cite[Section~2]{Sta92} (see
also our discussion in Section~\ref{subsec:enu}) that $\ell_1
(\Phi) = \xi_1 (\Phi)$ is equal to the number of interior vertices
of $\Gamma(\Phi)$. These vertices are exactly the positive roots
of $\Phi$ with support equal to $\Pi$ (i.e., the positive roots
which do not belong to any proper parabolic root subsystem
$\Phi_J$). The number of these roots was computed by Chapoton
\cite{Ch06} and admits an elegant, uniform formula; see
\cite[Proposition~1.1]{Ch06}. It would be interesting to find
uniform interpretations or formulas for $\ell_i (\Phi)$ or $\xi_i
(\Phi)$ for other values of $i$. We are not aware of a simple
closed form expression for $\ell_i (\Phi)$ in the type $A_n$ and
$B_n$ cases.

\medskip
2. It is natural to inquire for a more conceptual proof of
Proposition~\ref{prop:Dn}, in the spirit of those of
Propositions~\ref{prop:An} and \ref{prop:Bn}.

\medskip
3. The unimodality of the derangement polynomials was first proved by
Brenti \cite[Corollary~1]{Bre90}, who also asked for a combinatorial proof 
\cite[p.~1140]{Bre90}. Such a proof was given by Stembridge 
\cite[Corollary~2.2]{Ste92}. Theorem~\ref{thm:bary} provides another 
combinatorial proof (for a stronger statement). Since the barycentric 
subdivision $\sd(2^V)$ is a regular subdivision of $2^V$, the unimodality 
of the derangement polynomials also follows from (\ref{eq:localbary}) 
and \cite[Theorem~5.2]{Sta92}.

\section*{Acknowledgments} The authors wish to thank the anonymous referees for
their helpful comments. The second author was co-financed by the European Union
(European Social Fund - ESF) and Greek national funds through the Operational
Program ``Education and Lifelong Learning" of the National Strategic Reference
Framework (NSRF) - Research Funding Program: Heracleitus II. Investing in
knowledge society through the European Social Fund.

\end{document}